\documentclass{amsart}%
\usepackage{amsfonts}
\usepackage{amssymb}
\usepackage{amsmath}
\usepackage{amsthm}
\usepackage{graphicx}%
\newtheorem{thm}{Theorem}
\newtheorem{example}{Example}
\newtheorem{lem}{Lemma}
\newtheorem{corollary}{Corollary}
\newtheorem{question}{Question}

\newtheorem{proposition}{Proposition}
\input{Macros.tex}
\begin{document}
\title{Polyhedra with simple dense geodesics}
\author{Jin-Ichi Itoh}
\author{Jo\"{e}l Rouyer}
\author{Costin V\^{\i}lcu}

\begin{abstract}
We characterize polyhedral surfaces admitting a simple dense geodesic ray and convex polyhedral surfaces
with a simple geodesic ray.
\end{abstract}

\maketitle

\section{Introduction}
 
A tetrahedron is said to be isosceles if its faces are congruent to each
other. Equivalently, the singular curvature of each of its vertices is exactly $\pi$. 
These simple objects have many interesting properties, especially from viewpoint of their intrinsic geometry. 
For instance, it is easy to see that they admit infinitely long simple geodesics, as well
as arbitrarily long simple closed geodesics. 
It was proved by V. Yu. Protasov \cite{P} that this latter property characterizes them among convex polyhedra;
his result has recently been strenghtened to convex surfaces by A. Akopyan and A. Petrunin \cite{AP}.
However, there exist non-convex polyhedra with this property, as for example flat tori.

The aim of this paper is to determine (non-necessarily convex) polyhedra having simple dense geodesics. 
It turns out that the objects we found had already been studied, for they play a key role in the theory of
rational billiards. 
In \cite{MT} they are named \textquotedblleft surfaces endowed with a 
\emph{flat structure (with parallel line field)}\textquotedblright. 
One can find many variants of the name in the literature, generally emphasizing the flatness. 
Since all the surfaces considered in this paper -- namely polyhedra -- are flat (with conical singularities), we chose
to simply call them \emph{parallel polyhedra}.

Our main result (Theorem \ref{TMain}) states that a polyhedron is parallel if and only if it is orientable and admits a simple %
dense geodesic ray. In this case, at any point, a ray starting in almost any direction will be simple and dense.
We prove, moreover (Theorem \ref{convex}), that a simple geodesic ray on a convex polyhedron $P$ is dense in $P$.

This paper is about intrinsic geometry of surfaces: when we write
\textquotedblleft polyhedron\textquotedblright, we always mean a surface. 
More precisely, a compact surface flat everywhere except at finitely many
singularities, of conical type. In other words, it is a $2$-dimensional
manifold obtained by gluing (with length preserving maps) finitely many
Euclidean triangles along their edges. The conical points are called
\emph{vertices}. The (\emph{singular}) \emph{curvature} of a vertex $v$,
denoted by $\omega(v)$, is defined as $2\pi$ minus the total angle around it.
This notion of polyhedron is more general than the notion of boundary
of solid polyhedra in $\mathbb{R}^{3}$, for it also includes non-orientable
manifolds. Of course, once a polyhedron that can be embedded in $\mathbb{R}^{3}$,
it can be embedded in many different ways.

When we say that a polyhedron is convex, we mean that it is homeomorphic to
the sphere and that all its vertices have positive curvature. It is a famous
result of A. D. Alexandrov that such a polyhedron admits a unique (up to
rigid motions) realization as the boundary of a convex polyhedron in
$\mathbb{R}^{3}$ -- in the usual sense -- whence the denomination. Notice that
such a surface is also isometric to the boundary of many more non-convex
polyhedra, see for instance \cite{code61} or \cite{IRV}.


\section{Prerequisite}

In this section we recall a few definitions and give some basic results destined to be used later.

Let $P$ be a polyhedron and $I\subset\mathbb{R}$ an interval. 
A geodesic is a map $G:I\rightarrow P$ such that for any $t\in I$ and any $s\in I$ close
enough to $t$, we have $d\left(  G\left(  s\right)  ,G\left(  t\right)  \right)
=\left\vert s-t\right\vert $. 
In particular, all geodesics are parametrized by arc-length. 
When $I=\mathbb{R}$, we speak of a \emph{geodesic line}, and 
when $I\neq\mathbb{R}$ is an unbounded interval, we speak of a \emph{geodesic ray}. 
A geodesic is said to be \emph{simple} if it admits no proper self-intersection. 
That is, there is no pair of instants $t_{1}$, $t_{2}$ such that 
$G\left( t_{1}\right) =G\left( t_{2}\right)$ 
but
$G^{\prime}\left( t_{1}\right) \neq G^{\prime}\left( t_{2}\right)$, 
where
$G^{\prime}\left( t\right)$ represents the direction of $G$ at point $G\left( t\right)$. 
According to this definition, a periodic geodesic line may
be simple, though not injective; we call it a \emph{simple closed geodesic}.

It is well known and easy to see that a geodesic cannot pass through a
positively curved vertex (see for instance \cite{BCS}), but it may pass
through a negatively curved one. Indeed, a broken line through a vertex $v$
will be locally minimizing if and only if it separates the space of directions at $v$ into two parts of measure at least $\pi$. It follows that geodesics may
have branch points. So, there in no good notion of geodesic flow on a
polyhedron. A geodesic which avoids all vertices will be called \emph{strict}.

We denote by $V\left( P\right)$ the set of all vertices of $P$, and  by
$P^{\ast}$ the open (Riemannian) flat manifold $P\setminus V\left( P\right)$. 
So $P^{\ast}$ carries a natural notion of parallel transport. 
The parallel transport along a curve $\gamma:I\rightarrow P^{\ast}$
will be denoted by $||_{\gamma}$.

Although these polyhedra are not \emph{Alexandrov spaces with bounded curvature}, 
small enough balls clerly are either Alexandrov spaces with
curvature bounded below, or Alexandrov spaces with curvature bounded above, in
both cases by $0$ (for definitions and basic properties, see for instance
\cite{BBI}). It follows that the definition of \emph{space of directions} of
Alexandrov spaces applies here. We denote by $\Sigma_{p}$ the space of
directions at point $p\in P$ (it is a circle); 
for any subset $Q$ of $P$, $\Sigma Q$ stands for the disjoint union $\coprod_{p\in Q}\Sigma_{p}$. 
Hence $\Sigma P^{\ast}$ is naturally identified to the unit tangent bundle over $P^{\ast}$. For $p\in
P^{\ast}$, we denote by $\Delta_{p}$ the projective line obtained as the
quotient of $\Sigma_{p}$ by the group $\left\{  \pm id\right\}  $, and by
$\Delta P^{\ast}$ the corresponding bundle over $P^{\ast}$. A section of
$\Delta P^{\ast}$ will be called a \emph{line distribution} on $P^{\ast}$. 
It is said to be \emph{parallel} if its integral curves form a geodesic foliation
and, restricted to any small domain, those integral lines become parallel
once the domain is unfolded onto a plane.

Let $u\in\Sigma_{p}$. The maximal strict geodesic starting at $p$ in direction
$u$ is denoted by $\gamma_{u}$. If $\gamma_{u}$ is not defined on $\left[
0,\infty\right[  $, \ie, if $\gamma_{u}$ meets some vertex in the positive
direction, then $u$ is said to be \emph{singular}.

\bigskip The distance between two points $p$, $q\in P$ is denoted by $d\left(
p,q\right)  $. For $r>0$, $B\left(  p,r\right)  $ stands for the open ball of
radius $r$ centered at $x$.

We start with a few simple lemmas.

\begin{lem}
\label{L_ap}
Let $[a_{1}b_{1}]$ and $[a_{2}b_{2}]$ be two segments of length
$2l$ in the standard Euclidean plane. Let $m_{i}$ be the midpoint of
$[a_{i}b_{i}]$. Assume that $m_1\ne m_2$ and that $a_{1}$ and $a_{2}$ are in the same half plane
bounded by the line through $m_{1}$ and $m_{2}$. Put $\alpha_{i}%
=\measuredangle\left(  \overrightarrow{m_{1}m_{2}},\overrightarrow{a_{i}b_{i}%
}\right)  $ and $\delta=\alpha_{2}-\alpha_{1}$.
Then the segments intersect if and only if
\[
d\left(  m_{1},m_{2}\right)  \max\left(  \left\vert \sin\alpha_{1}\right\vert
,\left\vert \sin\alpha_{2}\right\vert \right)  \leq l\left\vert \sin
\delta\right\vert \text{.}%
\]
In particular, if the segments do not intersect, then
\[
\left\vert \sin\delta\right\vert <\frac{d\left(  m_{1},m_{2}\right)  }%
{l}\text{.}%
\]
\end{lem}

The proof is elementary and is left to the reader.

\begin{lem}
\label{L_LLD}
Assume there exists on $P$ a simple dense ray $G$. Then
there exists a locally Lipschitz line distribution $D$ on $P^{\ast}$ such that
for any $t$, $D_{G\left(  t\right)}$ is tangent to $G$.
\end{lem}

\begin{proof}
First define $D$ on $\mathrm{Im}\left(  G\right)  $ in the obvious way. Denote
by $P^{\varepsilon}$ the compact set $P\setminus\bigcup_{v\in V\left(
P\right)  }B\left(  v,\varepsilon\right)  $. From Lemma \ref{L_ap}, there is a
constant $K_{\varepsilon}$ depending only on $\varepsilon$ such that $D$ is
$K_{\varepsilon}$-Lipschitz continuous on $P^{\varepsilon}\cap\mathrm{Im}%
\left(  G\right)  $. Since $\mathrm{Im}\left(  G\right)  $ is dense, $D$
admits a unique $K_{\varepsilon}$-Lipschitz continuous extension
to $P_{\varepsilon}$ for any $\varepsilon$, and consequently
an unique continuous extension to $P^{\ast}$. This extension is obviously
locally Lipschitz continuous.
\end{proof}

\begin{lem}
\label{L_Gb2P}Let $p$, $q$ be two (possibly coinciding) points on $P$. Then,
the set $C\subset\Sigma_{p}$ of those $u$ such that $\gamma_{u}$ meets $q$ is
at most countable.
\end{lem}

\begin{proof}
It is sufficient to prove that, for any $a\in\mathbb{\mathbb{N}}%
^{\mathbb{\mathbb{\ast}}}$, the set of directions $u$ such that $\gamma
_{u}|[0,a]$ meets $q$ is finite. Assume on the contrary that there are
infinitely many distinct $u_{n}\in\Sigma_{x}$ such that $\gamma_{u_{n}}$ meets
$q$ at time $t_{n}\leq a$. By the Ascoli theorem, one can extract from $\left\{
\gamma^{n}\overset{\mathrm{def}}{=}\gamma_{u_{n}}|[0,t_{n}]\right\}_{n}$ a
subsequence (still denoted by $\gamma^{n}$) converging to a curve $\gamma$
from $p$ to $q$. Let $\tau_{n}$ be the supremum of those $t>0$ such that
$\gamma^{n}|]0,t]$ does not intersect $\operatorname{Im}\left(  \gamma
^{n+1}\right)  $, and $\tau_{n}^{\prime}$ be such that $\gamma^{n+1}\left(
\tau_{n}^{\prime}\right)  =\gamma\left(  \tau_{n}\right)  \overset
{\mathrm{def}}{=}p_{n}$. It is clear that $\tau_{n}$ and $\tau_{n}^{\prime}$
are bounded by $a$, and on the other hand, cannot approach $0$. Hence, the
geodesic digon $D_{n}\overset{\mathrm{def}}{=}\operatorname{Im}\left(
\gamma^{n}|[0,\tau_{n}]\right)  \cup\operatorname{Im}\left( \gamma
^{n+1}|[0,\tau_{n+1}]\right)$ tends to a geodesic arc of $\gamma$. It follows that
the sum $s_{n}$ of the angles of $D_{n}$ must tend to $0$. By the Gauss-Bonnet
formula, this sum equals the sum of the curvatures of the vertices included
inside the digon, and consequently $s_{n}$ may take only finitely many
distinct value. Hence $s_{n}=0$ for $n$ large enough. Since $u_{n}$ and
$u_{n+1}$ were supposed to be distinct, we get a contradiction.
\end{proof}

\begin{corollary}
\label{C_SD}For any $p\in P^{\ast}$, the set $S_{p}\subset\Sigma_{p}$ of
singular directions is at most countable.
\end{corollary}

\begin{lem}
\label{L_WDGT}Let $\sigma$ be a segment on $P$, $p$ a point of $\sigma$ and
$u\in\Sigma_{p}$ transverse to $\sigma$. For $y\in\sigma$, denote by $u_{y}\in\Sigma_y$
the parallel transport of $u$ along $\sigma$. Then, the set of those points
$y\in\sigma$ such that $u_{y}$ is singular is at most countable.
\end{lem}

\begin{proof}
The proof is similar to the one of Lemma \ref{L_Gb2P}. It is sufficient to
prove that, for any $a\in\mathbb{\mathbb{N}}^{\mathbb{\mathbb{\ast}}}$, the
set of points $y$ such that $\gamma_{u_{y}}$ is not well defined on $\left[
0,a\right]$ is finite. Assume on the contrary that there are infinitely many
$y_{n}\in\sigma$ such that $\gamma_{u_{y_{n}}}$ meets a vertex at time
$t_{n}\leq a$. Possibly passing to a subsequence, we can assume that it is the same
vertex for all those points. Put $\gamma^{n}\overset{\mathrm{def}}{=}%
\gamma_{u_{y_{n}}}|[0,t_{n}]$. By extracting a subsequence, one can assume
that $\gamma^{n}$ is converging to a curve $\gamma$. Let $q_{n}$ be the first
intersection point of $\gamma^{n}$ and $\gamma^{n+1}$ along $\gamma^{n}$. Let
$T_{n}$ be the triangle $T_{n}$ whose vertices are $y_{n}$, $y_{n+1}$ and
$q_{n}$, and the sides are parts of $\gamma^{n}$, $\gamma^{n+1}$ and $\sigma$.
By the Gauss-Bonnet theorem, the angle of $T_{n}$ at $q_{n}$ equals the sum of the
curvatures of the vertices included in $T_{n}$, and so, it can take only finitely
many distinct values. On the other hand, since $\gamma^{n}$ tends to $\gamma$,
this angle should tend to zero. Hence, it must vanish for large $n$, in
contradiction with the fact that $\gamma^{n}$ and $\gamma^{n+1}$ are distinct.
\end{proof}


\section{Parallel polyhedra}

In this section we introduce the notion of parallel polyhedra, and give their basic properties.

If the holonomy group of $P^{\ast}$ is either trivial, or equal to $\left\{
id,-id\right\}$, $P$ is said to be \emph{parallel}. In this case, the
parallel transport of lines does not depends on the path. In other words, there
exist natural bijections $\tau_{p}^{q}:\Delta_{p}\rightarrow\Delta_{q}$
such that, for any path $\gamma$ from $p$ to $q$ and any tangent line
$l\in\Delta_{p}$, the parallel transport of $l$ along $\gamma$ is $\tau
_{p}^{q}\left( l\right)$. It follows that there is a well defined notion a
line direction that does not depends on the point $p\in P^{\ast}$. We denote
by $\tilde{\Delta}P^{\ast}$ the set of line directions on $P^{\ast}$.

Examples of parallel polyhedra are given after Proposition \ref{p_PLD2P}.

In order to give a first characterization of parallel polyhedra, we need the
following lemma.

\begin{lem}
\label{L_2PLD2O}
If $P$ admits two parallel line distributions that are not
orthogonal at some point, then it is orientable.
\end{lem}

\begin{proof}
Let $D$ and $D^{\prime}$ be such distributions and assume that $P$ is not
orientable. So there exists a loop $\gamma:[a,b]\rightarrow P^{\ast}$ such that
$||_{\gamma}$ is a reflection. Let $\tilde{\gamma}:[a,b]\rightarrow P^{\ast}$
be a path homotopic to $\gamma$ in $P^{\ast}$, consisting in finitely many
segments which are either parallel or normal to $D$. Since $P^{\ast}$ is flat,
$||_{\tilde{\gamma}}=||_{\gamma}$, whence $||_{\gamma}D_{\gamma\left(
0\right)  }=D_{\gamma\left(  0\right)  }$. The same holds for $D^{\prime}$,
whence $D$ and $D^{\prime}$ are either equal or orthogonal.
\end{proof}

\begin{proposition}
\label{p_PLD2P}
The following statements are equivalent.
\begin{enumerate}
\item $P$ is parallel.

\item $P$ is orientable and admits a parallel line distribution.

\item $P$ admits two parallel line distributions that are not orthogonal at some point.
\end{enumerate}
\end{proposition}

\begin{proof}
Assume that $P$ is parallel. If $P$ was not orientable, then its holonomy
group would contain a reflection, in contradiction with the definition of
parallel polyhedra. Choose $p\in P^{\ast}$ and $l\in\Delta_{p}P^{\ast}$. The
line distribution is given by $D_{q}=\tau_{p}^{q}\left(  l\right)  $.

Conversely, assume that $P$ is orientable and let $D$ be a parallel line distribution.
Choose a piecewise smooth loop $\gamma:[0,1]\rightarrow P^{\ast}$ with
basepoint $p\overset{\mathrm{def}}{=}\gamma\left(  0\right)  =\gamma\left(
1\right)  $, and $u\in\Sigma_{p}$. Denote by $\tau_{t}:\Sigma_{p}%
\rightarrow\Sigma_{\gamma\left(  t\right)  }$ the parallel transport along
$\gamma|[0,t]$. Since $D$ is parallel, the (mod $\pi$) angle $\measuredangle\left(
D_{\gamma\left(  t\right) },\pm\tau_{t}\left( u\right) \right)$ is
constant with respect to $t$, whence $\tau^{1}\left(  u\right)  \in\left\{
\pm u,\pm S_{D_{p}}\left(  u\right)  \right\}  $, where $S_{l}$ stands for the
reflection with respect to $l$. Now, since $P$ is orientable, $\tau^{1}=\pm
id$. Lemma \ref{L_2PLD2O} ends the proof.
\end{proof}

\begin{example}
\label{ExIT}Any isosceles tetrahedron is parallel. Figure \ref{F1} shows an
unfolding and a parallel line distribution. By Proposition \ref{P_kpi} below,
(possibly degenerated) isosceles tetrahedra are the only parallel convex polyhedra.
\end{example}

\begin{figure}[ptb]
\begin{center}
\includegraphics[width=.3\textwidth]{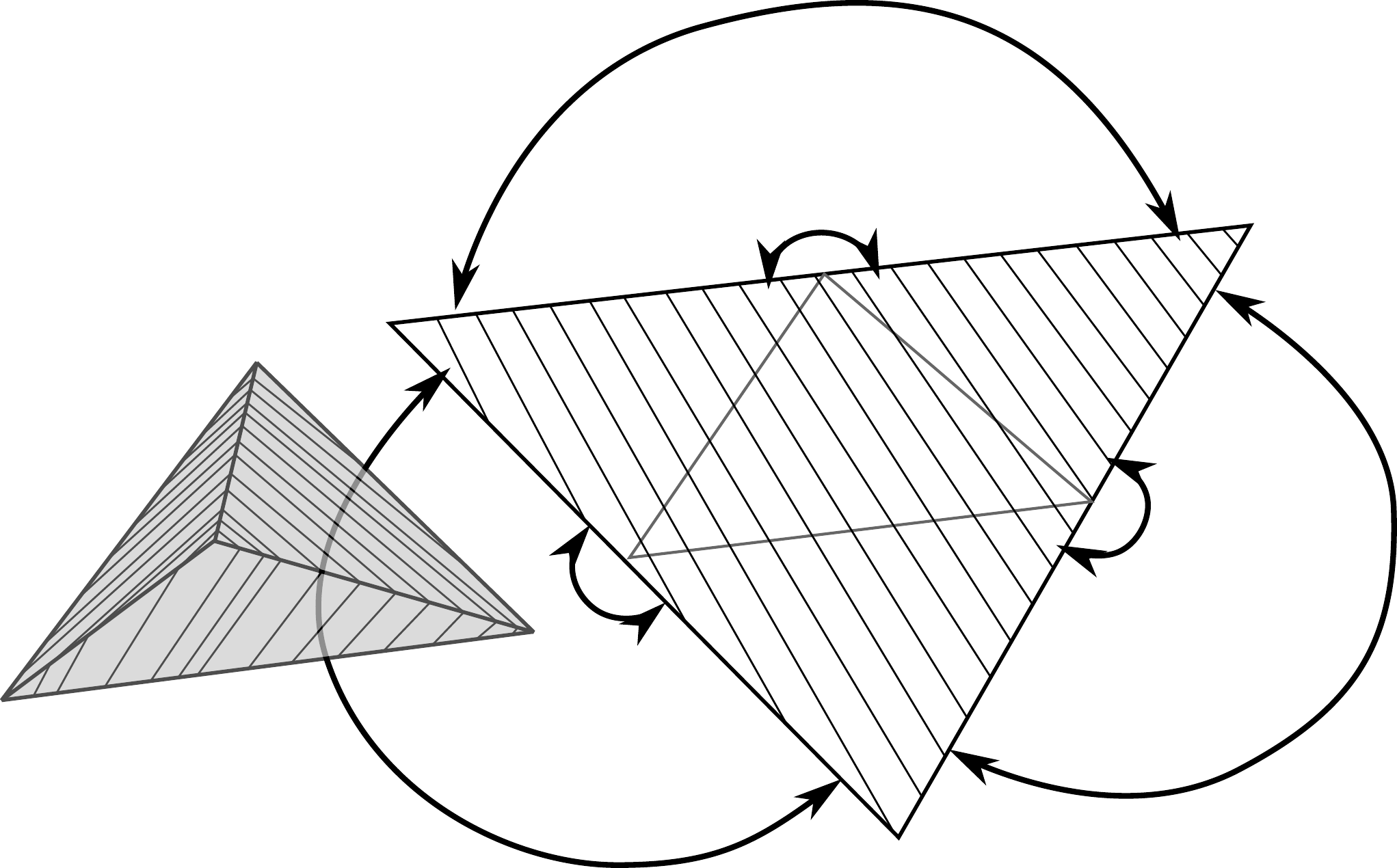}
\end{center}
\caption{Isoceles tetrahedra with a parallel line distribution.}%
\label{F1}%
\end{figure}

\begin{example}
\label{ExFT}Any flat torus is clearly a parallel polyhedron.
\end{example}

\begin{example}
\label{ExRD}Consider a polygonal domain $\Delta$ in $\mathbb{R}^{2}$ such that each
segment of its boundary is parallel either to the $x$-axis or the $y$-axis.
Glue two copies of $\Delta$ alomg their boundaries, by identifying the corresponding points.
Then the obtained polyhedron (called the double of $\Delta$) is parallel.
\end{example}

\begin{proposition}
\label{P_kpi}
If $P$ is parallel, the curvature of any vertex belongs to $\mathbb{Z}\pi$.
\end{proposition}

\begin{proof}
Let $\gamma:[0,1]\rightarrow P$ be a Jordan polygonal curve enclosing a vertex $v$. 
Let $s$ be the sum of its $n$ angles, measured toward the domain containing $v$. 
On one hand, $\measuredangle\left( \dot{\gamma}\left( 0\right) ,||_{\gamma}\dot{\gamma}\left( 0\right) \right) =n\pi-s$, 
and on the other hand, by the Gauss-Bonnet theorem, $\omega\left( v\right)=s-(n-2)\pi$. 
Since $P$ is parallel, $\measuredangle\left( \dot{\gamma}\left( 0\right) ,||_{\gamma}\dot{\gamma}\left( 0\right) \right)$
equals $0$ or $\pi$, and the conclusion follows.
\end{proof}

There is no converse of Proposition \ref{P_kpi} (see Example \ref{ExNP} below), except in the case of
polyhedra homeomorphic to the sphere.

\begin{proposition}
If $P$ is homeomorphic to the sphere and all vertices have curvature in $\mathbb{Z}\pi$ then $P$ is parallel.
\end{proposition}

\begin{proof}
Let $\gamma:[0,1]\rightarrow P^{\ast}$ be a closed loop. 
The parallel transport along $\gamma$ is a rotation of angle $\alpha=2\pi-\omega\left( D\right)$, 
where $D$ is the domain included on the left side of $\gamma$. 
Since by hypothesis, the curvature of each vertex is divisible by $\pi$, so is $\alpha$. 
Consequently, the holonomy group of $P$ can contain only $id$ and $-id$, that is, $P$ is parallel.
\end{proof}

\begin{example}
\label{ExNP}
Consider the double of the gray polygonal domain shown in Figure \ref{F2}. 
Each vertex has curvature either $\pi$ or $-\pi$.
However, the parallel transport along the curve $\gamma$ is a rotation of angle $\pi/2$.
\end{example}

\begin{figure}[ptb]
\begin{center}
\includegraphics[width=.3\textwidth]{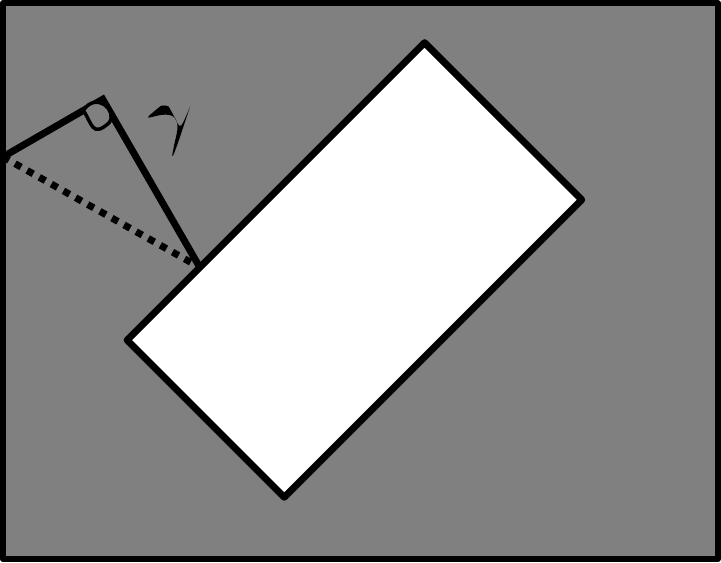}
\end{center}
\caption{Example of a non-parallel surface with vertices of curvature $\pm\pi$.}%
\label{F2}%
\end{figure}

\begin{proposition}
\label{p_AGS}On a parallel polyhedron, all strict geodesics are simple.
\end{proposition}

\begin{proof}
If there is a non-simple well defined geodesic, one portion of it is a (well
defined) closed geodesic arc $\gamma:[0,a]\rightarrow P^{\ast}$ making an
angle $\alpha\not \equiv 0~[\pi]$ at its base point. But, since $\gamma$ is a
geodesic $\dot{\gamma}\left(  a\right)  =||_{\gamma}\dot{\gamma}\left(
0\right)  \in\left\{  \pm\dot{\gamma}\left(  0\right)  \right\}  $ -- for $P$
is parallel -- and we get a contradiction.
\end{proof}

We have already mentioned that parallel polyhedra have been studied because of
their relation with rational billiards. One result of special interest for us
is the following lemma.

\begin{lem}
\label{L_MT} \cit{MT} 
There is a countable set $C\subset\tilde{\Delta}P^{\ast}$ such that, for any $u\in\Sigma P^{\ast}$ whose direction does not belong to
$C$, $\gamma_{u}$ is a simple dense ray.
\end{lem}


\section{Simple dense geodesics}

In this section we prove our main result.

\begin{lem}
\label{L_PLD}If $P$ admits a simple dense strict geodesic ray $G$, it admits a
parallel line distribution.
\end{lem}

\begin{proof}
By Lemma \ref{L_LLD}, $P$ admits a Lipschitz continuous line distribution $D$.
We shall show that it is actually parallel.

Clearly, the integral lines of this distribution are geodesics, for any arc of
such a line is limit of arcs of $G$. We claim that only finitely many integral
lines meet some vertex. Indeed two integral lines meeting one vertex $v$ must
form an angle at least $\pi$, otherwise an arc of $G$ through a point inside
the sector they determine should intersect one of them. Hence all but a final
number of integral lines are infinite in both direction, we call them
\emph{complete integral lines}. The union of all complete integral lines is
denoted by $C$.

Consider $x\in P^{\ast}$ and $B$ a ball centered at $P$ that does not contain
any vertex. Restricted to this ball, one can define Lipschitz unit vector
fields $\tau$, $\nu$ such that $\tau$ is parallel to $D$ and $\nu$ is normal
to $\tau$. Define $s_{p}\left(  y\right)  $ as the slope of $\tau_{p+y\nu_p}$ in
the basis $\left(  \tau_{p},\nu_{p}\right)  $. Set
\[
\phi\left(  p\right)  =\lim\sup_{h\rightarrow0}\left\vert \frac{s_{p}\left(
y\right)  }{y}\right\vert \text{.}%
\]

We claim that $\phi\left( p\right) =0$ almost everywhere, and so, $s_{p}$ is
derivable at $0$ and its derivative is $0$ for almost all $p$.

Take $p\in C$ and choose a sequence $y_{n}$ of real numbers, tending to $0$, 
such that $p_{n}\overset{\mathrm{def}}{=}p+y_{n}\nu_p$ belongs to $C$. 
Let $s$ (respectively $s_{n}$) be the segment of the integral line
through $p$ (respectively $p_{n}$) of lengh $2l$, whose midpoint is $p$
(respectively $p_{n}$). Then $s$ admits a neighborhood $N$ which, endowed with its
own intrinsic metric, is isometric to a neighborhood of a $2l$ long segment in
$\mathbb{R}^{2}$. For $n$ large enough, $s_{n}$ is included in $N$. Since
$s_{n}$ and $s$ don't intersect, we have $\left\vert \frac{s_{p}\left(
y_{n}\right)  }{y_{n}}\right\vert \leq\frac{1}{l}$. Hence, any adherence value
of $\left\vert \frac{s_{p}\left(  y_{n}\right)  }{y_{n}}\right\vert $ is at
most $\frac{1}{l}$, for arbitrarily large $l$, whence $\left\vert \frac
{s_{p}\left(  y_{n}\right)  }{y_{n}}\right\vert \rightarrow0$.

Now, we drop the assumption $p_{n}\in C$ and consider $q_{n}=p+z_{n}\nu_{p}\in
C$ such that $d\left(  p_{n},q_{n}\right)  <y_{n}^{2}$. Since $s_{p}$ is
$L$-Lipschitz continuous, we get
\begin{align*}
\left\vert \frac{s_{p}\left(  y_{n}\right)  }{y_{n}}-\frac{s_{p}\left(
z_{n}\right)  }{z_{n}}\right\vert  &  \leq\left\vert \frac{s_{p}\left(
y_{n}\right)  }{y_{n}}-\frac{s_{p}\left(  z_{n}\right)  }{y_{n}}\right\vert
+\left\vert \frac{s_{p}\left(  z_{n}\right)  }{y_{n}}-\frac{s_{p}\left(
z_{n}\right)  }{z_{n}}\right\vert \\
&  \leq L\left\vert \frac{y_{n}-z_{n}}{y_{n}}\right\vert +\left\vert
\frac{s_{p}\left(  z_{n}\right)  }{z_{n}}\right\vert \left\vert \frac
{y_{n}-z_{n}}{y_{n}}\right\vert \\
&  \leq\left(  L+2\varepsilon\right)  \left\vert y_{n}\right\vert
\rightarrow0\text{.}%
\end{align*}
It follows that $\left\vert \frac{s_{p}\left(  y_{n}\right)  }{y_{n}%
}\right\vert $ still tends to $0$, whence $\phi\left(  p\right)  =s_{p}%
^{\prime}\left(  0\right)  =0$ for any $p\in C$.

\begin{figure}[ptbh]
\begin{center}
\includegraphics[width=.3\textwidth]{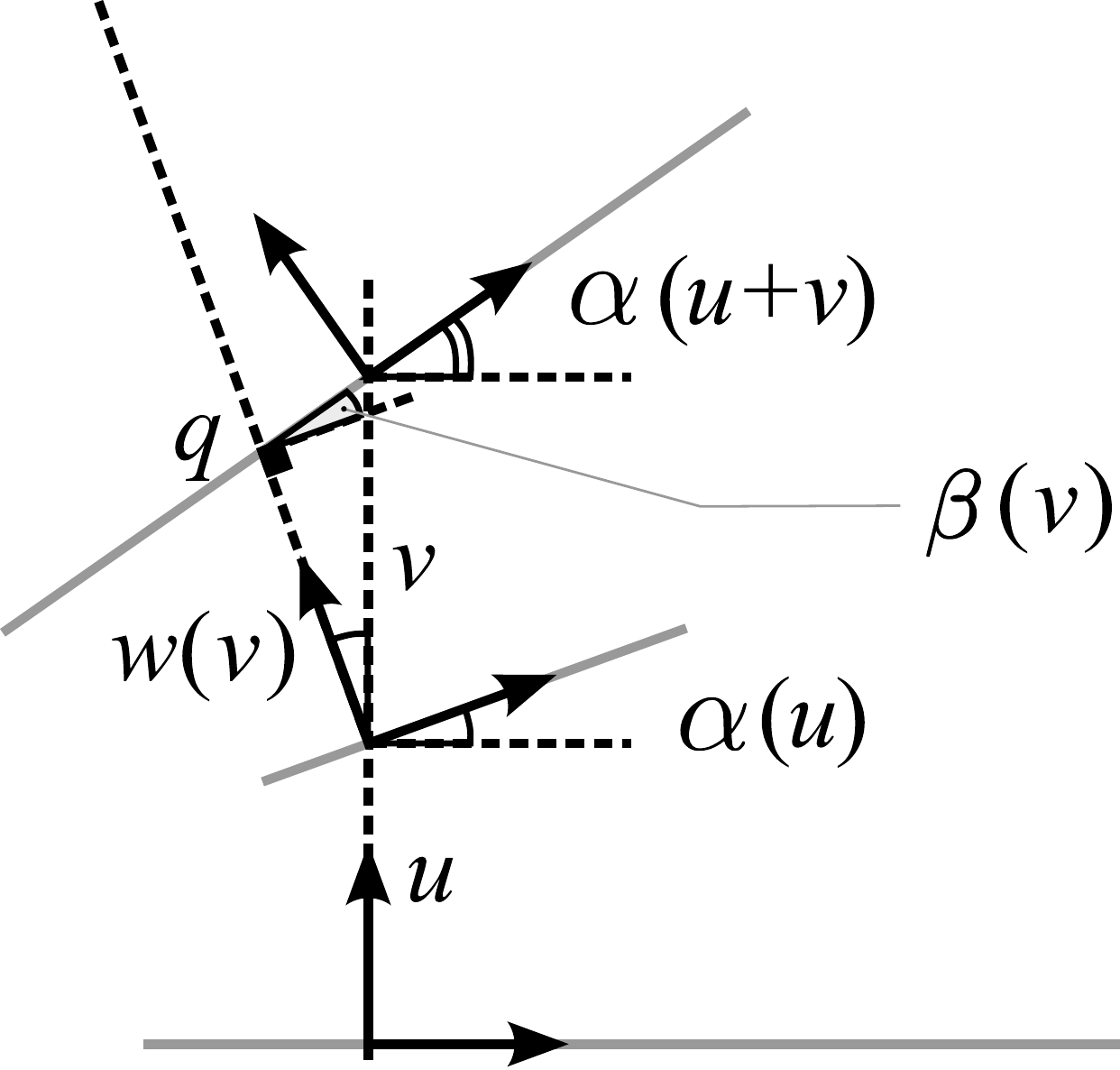}
\end{center}
\caption{Computation of $s_{p}^{\prime}\left(  u\right)  $ in the proof of
Lemma \ref{L_PLD}.}%
\label{F4}%
\end{figure}

Now, we prove that $s_{p}^{\prime}\left(  u\right)  =0$ for any $p$, $u$ such
that $p+u\nu_{p}\in C$. Let $\alpha\left(  y\right)  =\arctan s_{p}\left(
y\right)$, so%
\[
\alpha\left(  u+v\right)  -\alpha\left(  u\right)  =\arctan s_{p+uN}\left(
w\left(  v\right)  \right)  \overset{\mathrm{def}}{=}\beta\left(  v\right)\text{,}
\]
where $w\left(  v\right)  $ is the distance between $p+u\nu_{p}$ and the
intersection point $q$ between the lines $p+\left(  u+v\right)  \nu
_{p}+\mathbb{R\tau}_{p+\left(  u+v\right)  \nu_{p}}$ and $p+u\nu
_{p}+\mathbb{R\nu}_{p+u\nu_{p}}$ (see Figure \ref{F4}).

On the one hand $\lim\frac{\beta\left(  v\right)  }{w\left(  v\right)
}=s_{p+u\nu_{p}}^{\prime}\left(  0\right)  =0$, on the other hand the law of
sines in the triangle $p+u\nu_{p}$, $p+\left(  u+v\right)  \nu_{p}$, $q$ gives%
\[
\frac{w\left(  v\right)  }{v}=\frac{\cos\left(  \beta\left(  v\right)
+\alpha\left(  u\right)  \right)  }{\cos\beta\left(  v\right)  }%
\underset{v\rightarrow0}{\rightarrow}\cos\alpha\left(  u\right)  \text{,}%
\]
whence%
\[
\frac{\alpha\left(  u+v\right)  -\alpha\left(  u\right)  }{v}=\frac
{\beta\left(  v\right)  }{w\left(  v\right)  }\frac{w\left(  v\right)  }%
{v}\underset{v\rightarrow0}{\rightarrow}0\text{.}%
\]

Hence, for any $p$ and almost all $u$ we have $s_{p}^{\prime}\left(  u\right)
=0$. Now, since $s_{p}$ is Lipshitz continuous, we have
\[
s_{p}\left(  y\right)  =s_{p}\left(  0\right)  +\int_{0}^{y}s_{p}^{\prime
}\left(  u\right)  du=0+0\text{ }%
\]
for any $p$, $y$, and therefore $D$ is parallel.
\end{proof}

Now, we are in a position to state the main result of the paper.

\begin{thm}
\label{TMain}The following statements are equivalent.

\begin{enumerate}
\item \label{P_P}$P$ is parallel.

\item \label{P_ESDR}$P$ is orientable and admits a simple dense geodesic ray.

\item \label{P_2R}$P$ admits two simple dense geodesic rays which are
not orthogonal to each other at some point.

\item \label{P_MSDR}For any $p\in P^{\ast}$ there exists a countable set $C \subset \Sigma_{p}$ such that 
for any $u\in\Sigma_{p} \setminus C$, $\gamma_{u}$ is a strict simple ray.
\end{enumerate}
\end{thm}

\begin{proof}
By Proposition \ref{p_PLD2P} and Lemma \ref{L_PLD}, (\ref{P_P}), (\ref{P_2R})
and (\ref{P_ESDR}) are equivalent. By Lemma \ref{L_MT}, (\ref{P_P}) implies
(\ref{P_MSDR}) which obviously implies (\ref{P_2R}) and (\ref{P_ESDR}).
\end{proof}


\section{Convex case}

The aim of this section is to supplement Theorem \ref{TMain} in the convex
case by adding a new statement, namely the existence of a (not necessarily dense) simple geodesic ray.

\begin{lem}
\label{lem1}Let $P$ be a convex polyhedron. There is a finite set $F$ such that, 
for any simple geodesic $G$ of $P$ and any two arcs of $G$ lying on the same face of $P$, the angle between them belongs to $F$.
\end{lem}

\begin{proof}
The two arcs can be seen as the external parts of a longer arc of $G$. Joining
the endpoints of this arc by a segment produces a geodesic digon. 
By the Gauss Bonnet theorem, the sum of its angles, that is, the angle between the arcs, 
equals the curvature included in the digon, and so belongs to 
$F=\left\{ \sum_{v\in W}\omega\left( v\right) |W\subset V\right\}$, 
where $V$ denotes the set of vertices of $P$.
\end{proof}

Let $G$ be a simple geodesic ray of $P$. A point of $x=G\left(  t\right)  $ is
said to be of the \emph{first kind}, if there exists arbitrary small
$\varepsilon>0$ such that the intersection $\mathrm{Im}\left(  G\right)  \cap
B\left(  x,\varepsilon\right)$ is arcwise connected. It is said to be of
the \emph{second kind} if there exits points $x_{n}=G\left(  t_{n}\right)  $
and $x_{n}^{\prime}=G\left(  t_{n}^{\prime}\right)  $ such that $t_{n}$ and
$t_{n^{\prime}}$ tend to infinity, $x_{n}$ and $x_{n}^{\prime}$ tend to $x$,
and are locally separated by the arc of $G$ through $x$. The point $x$ is said
to be of the \emph{third kind} if (a) there exist points $x_{n}=G\left(
t_{n}\right)  $ tending to $x$ while $t_{n}$ tends to infinity, and (2) there
exits $\varepsilon>0$ such that the intersection of one of the two open halves of $B\left( x,\varepsilon\right)$ delimited by the arc of $G$ through $x$ does
not intersect $G$. Let $K_{i}\left(  G\right)  \subset\mathbb{R}^{+}$ be the
set of $t$ such that $G\left(  t\right)  $ is of the $i^{th}$ kind. It is easy
to see that $\mathbb{R}^{+}=\bigcup_{i=1,2,3}K_{i}\left(  G\right)  $.

\begin{lem}
\label{L_CK}If $G$ is a simple geodesic ray on a convex polyhedron $P$, then
all its points have the same kind.
\end{lem}

\begin{proof}
By Lemma \ref{L_ap}, if $G$ is simple and $G\left(  t_{n}\right)$ converges
to $G\left(  t\right)$ then we have the convergence of arcs $G\left(
[t_{n}-\varepsilon,t_{n}+\varepsilon]\right)  \rightarrow G\left(
[t_{n}-\varepsilon,t_{n}+\varepsilon]\right)  $. It follows that $K_{i}\left(
G\right)  $ is open, and the conclusion follows from the connectedness of $\mathbb{R}^{+}$.
\end{proof}

Hence, one can speak of the kind of a simple ray.

\begin{lem}
\label{L_K1}Let $G$ be a simple geodesic ray of first kind on a convex
polyhedron $P$, then there exists a simple ray of second or third kind.
\end{lem}

\begin{proof}
Let $p$ be an accumulation point of the sequence $\left\{  G\left(  n\right)
\right\}  _{n\in\mathbb{N}}$. By Lemma \ref{lem1}, for large $n$, arcs of $G$
through $G\left(  n\right)  $ are parallel one to another. Let $S$ be a segment through
$p$, parallel to $G$. There is at least one side of $S$ where arcs of $G$
accumulate; we call it the side of $G$. We can prolong $S$ in the direction of
$G$ as a quasi-geodesic $\overline{S}$ with angle $\pi$ on the side of $G$ as
long as it does not meet any vertex of curvature at least $\pi$. But if
$\overline{S}$ meets a vertex of curvature more than $\pi$, then $G$ should
intersect in the vicinicy of this vertex, which is impossible. If
$\overline{S}$ meets a vertex of curvature $\pi$ after one or several vertices
of curvature less than $\pi$, then $G$ should self-intersect in the vicinity
of the last of those vertices (see Figure \ref{F6}). Hence, either $G$ is infinitely prolongable, or it meets a $\pi$-curved vertex before any other. In this case, $G$ accumulates on the other side of $S$ too, in the
opposite direction (see Figure \ref{F6}).

So we can define $\overline{S}$ in the other direction. Once again, either
$\overline{S}$ is infinitely prolongable, or it meets a $\pi$-curved vertex.
In this case the prolongation of any arc of $G$ near $p$ should be a closed
geodesic, and we get a contradiction.

It follows that $\overline{S}$ can be prolongated as a ray, at least in one
direction. Note that $\overline{S}$ cannot have any self-intersection, for it
is parallel to $G$ which is simple. If $\overline{S}$ were periodical then,
once again, arcs of $G$ would prolong as closed geodesics. Hence $\overline
{S}$ visits at most once each vertex, and so, after leaving the last one,
becomes a geodesic ray. Since $G$ is parallel to $\overline{S}$ and acumulate
along it, then $\overline{S}$ cannot be of the first kind.
\end{proof}

\begin{figure}[ptb]
\begin{center}
\includegraphics[width=.3\textwidth]{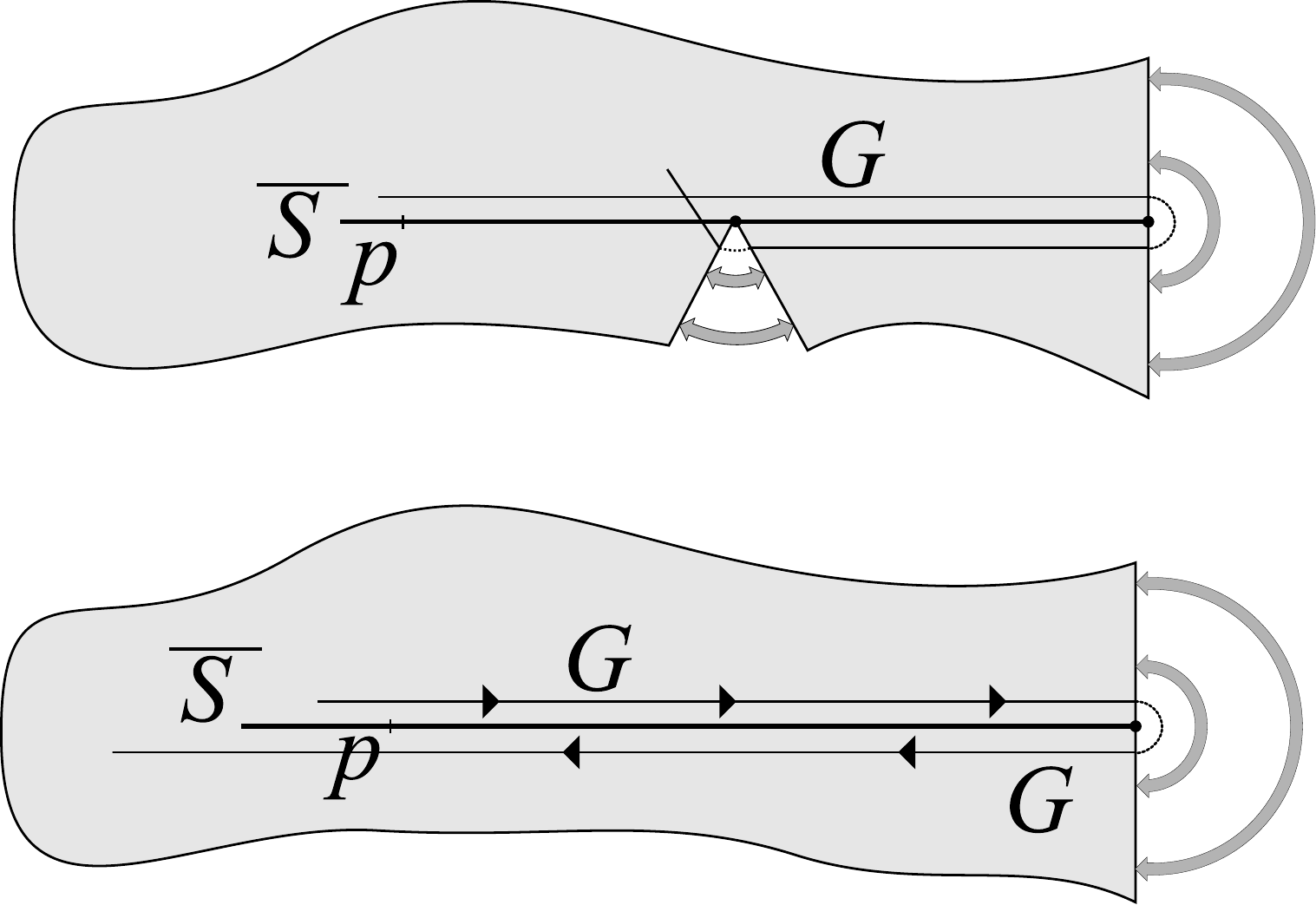}
\end{center}
\caption{Proof of Lemma \ref{L_K1}: if $\overline{S}$ meets a $\pi$-curved
vertex after another one, then $G$ should self-intersect. If it meets the $\pi
$-curved vertex first, then it can be defined in the other direction.}%
\label{F6}%
\end{figure}

\begin{lem}
\label{L_K2}On a convex polyhedron, there is no simple ray of third kind.
\end{lem}

\begin{proof}
Let $x=G\left(  t\right)  $ a point of third kind. By definition, there exists
$\varepsilon>0$ such that one half of $B_{0}=B\left( x,\varepsilon\right)$
does not intersect $G$, and there exist points $x_{n}=G\left( t_{n}\right)$ tending to
$x$ when $t_{n}$ tends to infinity. Let $A_{n}=G|[t_{n}-\varepsilon
,t_{n}+\varepsilon]$ and $A=G|[t-\varepsilon,t+\varepsilon]$. By Lemma
\ref{lem1}, those arcs are all parallel for large $n$. Equip $B$
with an orthonormal coordinate system such that $A\left(  t\right)  =\left(
t,0\right)  $ and each $A_{n}$ lies in the positive ordinate half plane. By
extracting subsequences, we can assume without loss of generality that (1)
the sequence $\left\{ t_{n}\right\} _{n}$ is increasing, (2) the sequence
$\left\{  d\left(  x_{n},A\right)  \right\}  _{n}$ is decreasing, and (3) the
arcs $A_{n}$ are all oriented in the same direction, say as the $x$-axis.

Prolong $A$ in the direction of the $x$-axis as a quasi-geodesic $\overline
{A}$ with angle $\pi$ on the $A_{n}$'s side, as long as $\overline{A}$ does
not meet any vertex of curvature greater than or equal to $\pi$. Indeed, if
it meets a vertex with curvature more than $\pi$, then the prolongation of
$A_{n}$ should self-intersect in the vicinity of this vertex, for $n$ large
enough. Assume now that $\bar{A}$ meets a $\pi$-curved vertex. Then either
$\overline{A_{n}}$ will self-intersect (in the case that $\overline{A}$ passes
through a vertex) or will intersect the half of $B_{0}$ of negative ordinates
(see Figure \ref{F5}). In both cases we get a contradiction. Hence
$\overline{A}$ is a quasi-geodesic ray.

Assume that $\overline{A}$ is parametrized in such a way that $\overline
{A}\left(  t\right)  =\left(  t,0\right)$ for $t$ small enough. Denote by
$\overline{A_{n}}$ the parametrization of $G$ such that $\overline{A_{n}%
}\left(  t\right)  =\left(  t,\varepsilon_{n}\right)  $ for small $t$. Put
$f_{n}\left(  t\right)  =d\left( \overline{A}\left( t\right) ,\overline
{A_{n}}\left(  t\right)  \right)$. We claim that, for $n$ large enough,
$f_{n}\left(  t\right)  \leq f_{n}\left(  0\right)  =\varepsilon_{n}$, with
equality for any $t$ such that $\overline{A}\left(  t\right) $ is not too
close from a $\pi$-curved vertex. Indeed, $f_{n}$ is constant as long as the
\textquotedblleft strip\textquotedblright\ delimited by $\overline{A}$ and
$\overline{A_{n}}$ contains no vertices. But if it happens that some vertex of
curvature distinct from $\pi$ is included in the \textquotedblleft
strip\textquotedblright, then $\overline{A_{n}}$ and $\overline{A}$ should
intersect in its vicinity, provided there are no other vertices around.
Moreover, the distance between the vertex and the intersection point depends
only on the curvature of the vertex and on $\varepsilon_{n}$, and tends to $0$
when $\varepsilon_{n}$ becomes smaller and smaller. Hence, by choosing $n$
large enough, we can ensure that no other vertex will interfere. If
$\overline{A_{n}}$ intersected $\overline{A}$, then it would also intersect
$\overline{A_{m}}$ for $m>n$, in contradiction to the simpleness of $G$ .
Hence only $\pi$-curved vertices may interfere. But in this case, it is easy to
see that $f_{n}\left(  t\right)  $ will become again constantly equal to
$\varepsilon_{n}$, when $\overline{A}$ leaves the vicinity of the vertex.

Now $A_{n+1}$ is a subarc of $\overline{A_{n}}$, meaning that $\overline
{A_{n}}$ will enter again into $B_{0}$ at ordinate $\varepsilon_{n+1}%
<\varepsilon_{n}$. By the claim, $\overline{A}$ will also enter again into
$B$, at ordinate $\varepsilon_{n+1}\pm\varepsilon_{n}$. Indeed, $\varepsilon
_{n+1}+\varepsilon_{n}$ is impossible, because it would imply that the strip
between $\overline{A}$ and $\overline{A_{n}}$ is twisted, in contradiction
to the orientability of $P$. It follows that $\overline{A}$ enters into
$B_{0}$ with an ordinate $y<0$. Let $m$ be large enough to ensure
$\varepsilon_{m}<-y$; due to the claim, $\overline{A_{m}}$ should enter too
into the negative ordinate part of $B_{0}$, and we get a contradiction. Hence
$K_{3}\left(  G\right)$ is empty.
\end{proof}

\begin{figure}[ptb]
\begin{center}
\includegraphics[width=.3\textwidth]{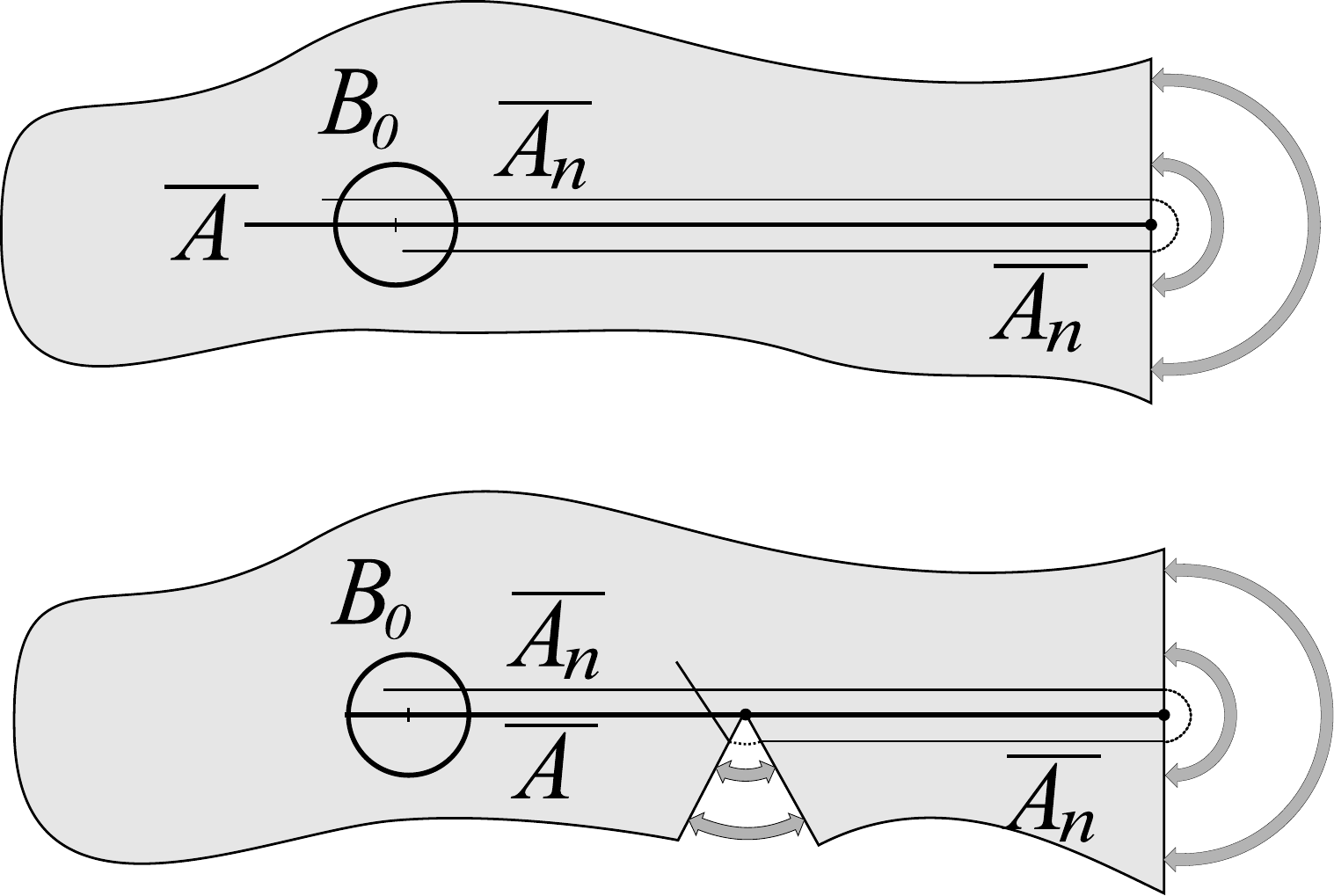}
\end{center}
\caption{Proof of Lemma \ref{L_K2}: if $\overline{A}$ meets a $\pi$-curved
vertex, then $\overline{A_{n}}$ should either self-intersect or enter the
negative ordinate part of $B_{0}$.}%
\label{F5}%
\end{figure}

\begin{thm}
\label{convex}
Let $G$ be a simple geodesic ray on a convex polyhedron $P$. Then $G$ is dense in $P$.
\end{thm}

\begin{proof}
By Lemma \ref{L_K2}, $G$ is not of the third kind. Assume first that $G$ is of
the second kind.

Assume that $\mathrm{cl} \left(  \operatorname{Im}\left( G\right) \right)$ is not the whole polyhedron 
and let $o$ be a point in $P\setminus\mathrm{cl} \left(\operatorname{Im}\left( G\right) \right)$. 
Let $B_{0}=B\left( o,d\left(o,\mathrm{cl}\left( S\right) \right) \right)$ and 
$x\in\partial \mathrm{cl} \left(S\right)  \cap\partial B_{0}$. 
By choosing $o$ close enough to a flat point of $\mathrm{cl} \left( \operatorname{Im}\left( G\right) \right)$, 
one can assume without loss of generality that $x$ is not a vertex. Choose
$\varepsilon>0$ smaller than the half-distance between $x$ and its closest vertex. 
Let $x_{n}=G\left( \tau_{n}\right)$ be a point in $\mathrm{\operatorname{Im}}\left( G\right)$ tending to $x$ such that
$\tau_{n}\rightarrow\infty$. Let $A_{n}=G|[\tau_{n}-\frac{\varepsilon}{2},\tau_{n}+\frac{\varepsilon}{2}]$. 
By Lemma \ref{lem1}, those arcs are all parallel for large $n$. 
Denote by $A$ the limit of $A_{n}$. 
Let $B$ be a ball centered at $p$, small enough to ensure that all $A_{n}$s that
intersect $B$ are parallel. Equip $B$ with an orthonormal coordinate system
such that $A\left(  t\right)  =\left( t,0\right)$ and $A_{n}$ lies in the
positive ordinate half plane. 
By Lemma \ref{L_K2}, all points of $G$ are of the
second kind, meaning that $x_{n}\notin A$, for otherwise there would be some
arcs of $G$ on both sides of $A$, and those with negative coordinates should
intersect $B_{0}$. Hence, by extracting subsequences, we can assume without
loss of generality that (1) the sequence $\left\{ \tau_{n}\right\} _{n}$
is increasing, (2) the sequence $\left\{  d\left(  x_{n},A\right)  \right\}
_{n}$ is decreasing, and (3) the arcs $A_{n}$ are all oriented in the same
direction, say as the $x$-axis.

Now, we get a contradiction using exactely the same construction as in the end
of the proof of Lemma \ref{L_K2}.

If $G$ were of the first kind, then we could apply the above argument to the
ray $\overline{S}$ provided by Lemma \ref{L_K1}. Hence $\overline{S}$ would be
dense, and then $G$ too; in contradiction to the fact that it is of
first kind.
\end{proof}


\section{Examples and questions}

\textbf{1.} As mentioned in the introduction, isosceles tetrahedra are the
only convex polyhedra admitting an unbounded length spectrum. In the light of
our result, one can ask if a polyhedron with an unbounded length spectrum
should be parallel. This question seems especially natural knowing that any
parallel polyhedron admits such a length spectrum \cite{M2}. The answer,
however, is negative, as shown by the following example.

\begin{example}
Let $a\in\left]  0,1\right[  $ be an irrational number. Consider the double
$D$ of the unit square $[0,1]\times\lbrack0,1]$. As a double of rectangle, $D$
is a degenerate isosceles tetrahedron, and so is parallel. Cut $D$ along the
segment $\sigma$ of the upper face (say) corresponding to $\left\{  a\right\}
\times\lbrack1/3,2/3]$. We obtain a manifold with boundary. Glue any polygon
of perimeter $2/3$ along its boundary in order to obtain a polyhedron $P$
homeomorphic to the sphere. $P$ is clearly not parallel, for it has vertices
whose curvature is not divisible by $\pi$. There is a natural injection
$i:D\setminus\sigma\rightarrow P$. Let $\gamma_{n}$ be the maximal geodesic of
$D$ starting at $\left(  a,0\right)  $ on the upper face (say), directed by
$\left(  1/n,1\right)  $. It is easy to see that $\gamma_{n}$ will never cross
$\sigma$, and so, is the image under $i$ of a closed geodesic of $D$ of length
$2n\sqrt{1+\frac{1}{n^{2}}}$.
\end{example}

\bigskip

\textbf{2. }The existence of only one simple dense geodesic ray does not
guarantee that $P$ is parallel, as illustrated by the following example. We do
not know if there exists non-orientable polyhedra with two (orthogonal) simple
dense rays.

\begin{example}
Let $a\in\lbrack0,1]\setminus\mathbb{Q}$. Consider the unit square
$[0,1]\times\lbrack0,1]$ with boundary identified as shown on Figure \ref{F3}:
$A$ on $A$, $B$ on $B$ and $C$ on $C$, respecting the orientation given by the
small white triangles. The resulting polyhedron has non-orientable genus $3$,
and a unique vertex of curvature $-2\pi$. It is easy to see that a line
starting at any point of abscissa $x\notin\mathbb{Q}+a\mathbb{Q}$ in the
direction of the $y$-axis is dense.
\end{example}

\begin{figure}[ptb]
\begin{center}
\includegraphics[width=.3\textwidth]{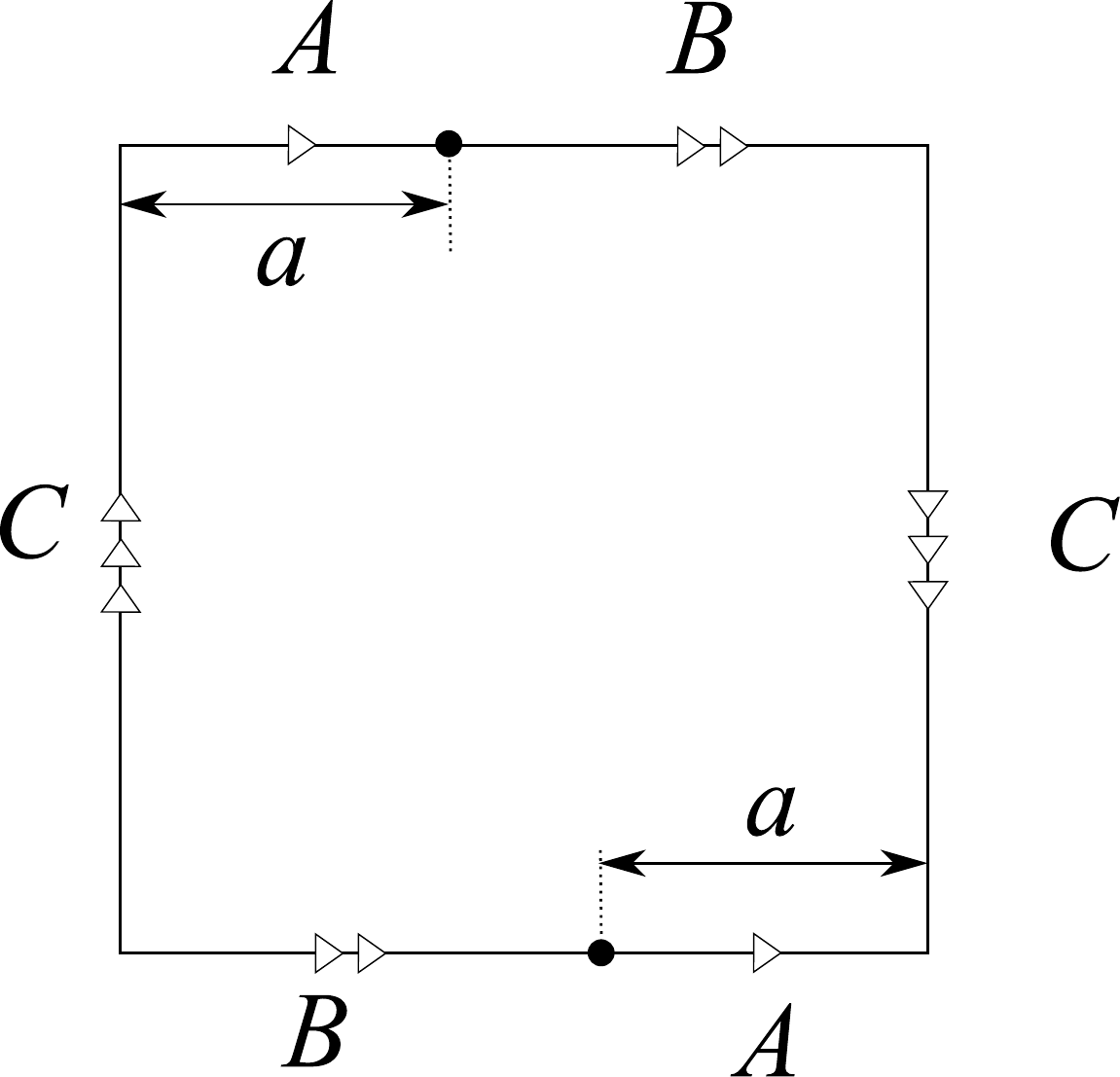}
\end{center}
\caption{Example of non-parallel polyhedron with simple dense geodesic ray.}%
\label{F3}%
\end{figure}

\bigskip

\textbf{3.} Without the convexity hypothesis, the existence of a simple
geodesic ray does not guarantee the parallelism.

\begin{example}
Let $P$ be a parallel polyhedron and $G$ a simple dense geodesic on $P$. Let
$\sigma$ be a segment disjoint of (and consequently parallel to)
$\operatorname{Im}\left(  G\right)  $. Cut $P$ along $\sigma$ and glue any
polygon whose perimeter equals the double of the length of $\sigma$. Denote by
$Q$ this new polyhedron and by $i:P\setminus\sigma\rightarrow Q$ the natural
injection. Then $i\circ G$ is a simple geodesic ray of $Q$, which is not dense
in $G$, for it avoids the glued polygon. Of course, $Q$ is not parallel.
\end{example}

\bigskip

\textbf{4. }We saw (Proposition \ref{p_AGS}) that on a parallel polyhedron all
strict geodesics are simple.

\begin{question}
Does the fact that all geodesics are simple characterizes parallel polyhedra?
\end{question}

\bigskip

\textbf{5.} V. Yu. Protasov \cite{P} showed that an unbounded length spectrum characterizes
isosceles tetrahedra among convex polyhedra, and his result has recently been strenghtened to convex surfaces 
by A. Akopyan and A. Petrunin \cite{AP}.
The following question remains open.

\begin{question}
Are isosceles tetrahedra the only convex surfaces with a simple dense geodesic ray? A
simple geodesic ray?
\end{question}

Acknowledgements. Jo\"{e}l Rouyer acknowledges support from the {\it Centre Francophone de Math\'ematique \`a
l'IMAR} and Costin V\^{\i}lcu acknowledges support from the {\it GDRI ECO-Math}.


\end{document}